\documentclass[10pt]{article}

\usepackage{amssymb}
\usepackage{amsmath}
\usepackage{amsthm}
\usepackage{mathrsfs}
\usepackage{ifthen}
\usepackage{tikz}

\usetikzlibrary{calc}
\usetikzlibrary{arrows}
\usetikzlibrary{arrows,decorations.pathmorphing,decorations.pathreplacing}

\usepackage{hyperref}
\hypersetup{colorlinks, citecolor=blue}

\newtheorem{thm}{Theorem}[section]
\newtheorem{lem}[thm]{Lemma}
\newtheorem{prop}[thm]{Proposition}
\newtheorem{cor}[thm]{Corollary}
\newtheorem{mainthm}{Theorem}

\theoremstyle{definition}
\newtheorem{defn}[thm]{Definition}

\newcommand{\belowcollar}[4]{
\draw[thick] (#1,#2) arc (0:180:#3 and #4);
\draw [thick,dashed]  (#1,#2) arc (360:180:#3 and #4);
}
\newcommand{\abovecollar}[4]{
\draw[thick,dashed] (#1,#2) arc (0:180:#3 and #4);
\draw [thick]  (#1,#2) arc (360:180:#3 and #4);}

\newcommand{\lazyllipse}[4]{
\draw[thick,thick] (#1,#2) arc (0:360:#3 and #4);}

\newcommand{\blankcircle}[2]{
  \fill[color=white] #1 circle (#2);}

\newcommand{\define}[1]{\textit{#1}}
\newcommand{\bk}[1]{{\langle #1 \rangle}}
\newcommand{\into}{\hookrightarrow}
\newcommand{\onto}{\twoheadrightarrow}
\newcommand{\bdy}{\partial}
\newcommand{\nth}{{\textrm{th}}} 
\def\ncl#1{\mathord{\langle}\mskip -4mu plus 0mu minus 0mu \mathord{\langle}#1\mathord{\rangle}\mskip -4mu plus 0mu minus 0mu \mathord{\rangle}}

\newcommand{\freegrp}{\mathbb{F}}
\newcommand{\free}{\freegrp}
\newcommand{\cl}[1]{{#1}^\star}
 
\newcommand{\Z}{\mathbb{Z}}
\newcommand{\Udder}{\ensuremath{\mathbb{U}}}

\newcommand{\fcs}{$\free$-conjugacy separable}
\newcommand{\fcsy}{$\free$-conjugacy separability}
\newcommand{\frcs}{freely conjugacy separable}
\newcommand{\frcsy}{free conjugacy separability}
\newcommand{\idouble}{\ensuremath{L}}

\begin{document}

\author{Larsen Louder \& Nicholas W.M. Touikan}

\title{Magnus pairs in, and free conjugacy separability of, limit groups}

\maketitle

\begin{abstract}
  There are limit groups having non-conjugate elements whose images
  are conjugate in every free quotient. Towers over free groups are
  freely conjugacy separable.
\end{abstract}

\section{Introduction}

This paper is concerned with the problem of finding free quotients of
finitely generated groups in which non-conjugate elements have
non-conjugate images. Given a finitely generated group $G$ which is
not a limit group, there is a finite collection of limit group
quotients $G\onto L_1,\dotsc,G\onto L_n$ of $G$ such that every
homomorphism $G\to\free$, where $\free$ is a free group, factors
through one of the factor groups $G\onto L_i$, hence it suffices to
consider the problem only for limit groups.

We are reduced then to the problem of finding free quotients of limit
groups in which non-conjugate elements have non-conjugate images. A
group is \emph{freely conjugacy separable}, or \fcs, if for any pair
$u,v \in G$ of non-conjugate elements there is a homomorphism to some
free group $G\to\freegrp$ such that the images of $u$ and $v$ in $\freegrp$
are non-conjugate.

We will give two different types of examples of limit groups which are
not \fcs\ for entirely different reasons. In Section
\ref{sec:magnus-pairs} we produce a limit group $L$ with elements
$u,v$ such that the cyclic groups $\bk{u},\bk{v}$ are non-conjugate,
but whose normal closures $\ncl u$ and $\ncl v$ coincide. We call such
a pair of elements a \define{Magnus pair} (see Definition
\ref{defn:Magnus-pair}.) Such elements must have conjugate images in
any free quotient by a theorem of Magnus \cite{Magnus-1931}. In
Section \ref{sec:other-case} we construct a limit group which is a
double of a free group over a cyclic group generated by a $C$-test
word (see Definition \ref{def:c-test}). These limit groups, called
\define{$C$-doubles}, are low rank and we are able to construct
their Makanin-Razborov diagrams encoding all homomorphisms into any
free group and directly observe the failure of free conjugacy
separability. This limit group was independently discovered and
studied by Simon Heil \cite{heil2016jsj}, who published a preprint
while this paper was in preparation. He uses this limit group to show
that limit groups are not freely subgroup separable.

\begin{defn}\label{defn:discriminating}
  A sequence of homomorphisms $\{\phi_i\colon G \to H\}$ is
  \define{discriminating} if for every finite subset $P \subset G
  \setminus \{1\}$ there is some $N$ such that for all $j \geq N, 1
  \not\in \phi_j(P)$.
\end{defn}

\begin{defn}
  A finitely generated group $L$ is a \define{limit group} if there is
  a \define{discriminating} sequence of homomorphisms $\{\phi_i\colon L \to
  \freegrp\}$, where $\freegrp$ is a free group.
\end{defn}

\begin{mainthm}\label{thm:not-conj-res-free}
  The class of limit groups is not \frcs.
\end{mainthm}

This should be seen in contrast to the fact that limit groups are
conjugacy separable \cite{C-Z-separable}. Furthermore Lioutikova in
\cite{Lioutikova-CRF} proves that iterated centralizer extensions (see
Definition \ref{defn:tower}) of a free group $\freegrp$ are \fcs. It
is a result of of Kharlampovich and Miasnikov \cite{KM-IrredII} that
all limit groups embed in to iterated centralizer extensions. Moreover
by \cite[Theorem 5.3]{gaglione2009almost} almost locally free groups
\cite[Definition 4.2]{gaglione2009almost} cannot have Magus
pairs. This class includes the class of limit groups which are
$\forall\exists$-equivalent to free groups. The class of iterated
centralizer extensions and the class of limit groups
$\forall\exists$-equivalent to free groups are contained in the class
of towers, also known as NTQ groups. We generalize these previous
results to the class of towers with the following strong \fcsy\
result:

\begin{mainthm}\label{thm:generic-sequence}
  Let $\freegrp$ be a non-abelian free group and let $G$ be a tower
  over $\freegrp$ (see Definition \ref{defn:tower}). There is a
  discriminating sequence of retractions $\{\phi_i\colon G \onto
  \freegrp\}$, such that for any finite subset $S \subset G$ of
  pairwise non-conjugate elements, there is some positive integer $N$
  such that for all $j \geq N$ the elements of $\phi_j(S)$ are
  pairwise non-conjugate in $\freegrp$. Similarly for any indivisible
  $\gamma \in L$ with cyclic centralizer there is some positive
  integer $M$ such that for all $k \geq M$, $r_k(\gamma)$ is
  indivisible.
\end{mainthm}

This theorem also settles \cite[Question 7.1]{gaglione2009almost},
which asks if arbitrarily large collections of pairwise nonconjugate
elements can have pairwise nonconjugate images via a homomorphism to a
free group. The proof of Theorem \ref{thm:generic-sequence} is in
Section \ref{sec:towers-fcs} and follows from results of Sela
\cite{Sela-Dioph-II} and Kharlampovich and Myasnikov \cite{KM-ift},
which form the first step in their (respective) systematic studies of
the $\forall\exists$-theory of free groups.

Finally, in Section \ref{sec:refinements} we analyze the failure of
\frcsy\ of our limit group with a Magnus pair and show that this is
very different from $C$-double constructed in Section
\ref{sec:other-case}. We then show that the \frcsy\ does not isolate
the class of towers within the class of limit groups.

Throughout this paper, unless mentioned otherwise, $\freegrp$ will
denote a non-abelian free group, $\freegrp_n$ will denote a
non-abelian free group of rank $n$, and $\freegrp(X)$ will denote the
free group on the basis $X$.

\section{A limit group with a Magnus pair}\label{sec:magnus-pairs}

Consider the of the fundamental group of the graph of spaces $\Udder$
given in Figure \ref{fig:1}.
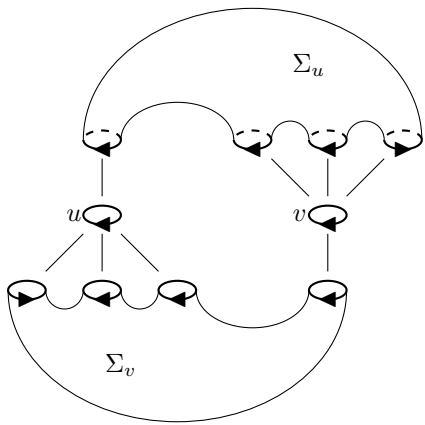
\begin{figure}[htb]
\centering
\begin{tikzpicture}[scale=0.5]
  \abovecollar{4}{2}{0.5}{0.25}
  \draw (3.5,1.75) node {${\blacktriangleleft}$};
  \abovecollar{8}{2}{0.5}{0.25}
  \draw (7.5,1.75) node {${\blacktriangleleft}$};
  \abovecollar{10}{2}{0.5}{0.25}
  \draw (9.5,1.75) node {${\blacktriangleleft}$};
  \abovecollar{12}{2}{0.5}{0.25}
  \draw (11.5,1.75) node {${\blacktriangleright}$};

  \lazyllipse{2}{-2}{0.5}{0.25}
  \draw (1.5,-2.25) node {${\blacktriangleright}$};
  \lazyllipse{4}{-2}{0.5}{0.25}
  \draw (3.5,-2.25) node {${\blacktriangleleft}$};
  \lazyllipse{6}{-2}{0.5}{0.25}
  \draw (5.5,-2.25) node {${\blacktriangleleft}$};
  \lazyllipse{10}{-2}{0.5}{0.25}
  \draw (9.5,-2.25) node {${\blacktriangleleft}$};

  \draw (3,2) arc (180:0:4.5 and 3.5);
  \draw (4,2) arc (180:0:1.5 and 1);
  \draw (8,2) arc (180:0:0.5 and 0.5);
  \draw (10,2) arc (180:0:0.5 and 0.5);

  \draw (10,-2) arc (360:180:4.5 and 3.5);
  \draw (9,-2) arc (360:180:1.5 and 1);
  \draw (5,-2) arc (360:180:0.5 and 0.5);
  \draw (3,-2) arc (360:180:0.5 and 0.5);
  
  \lazyllipse{4}{0}{0.5}{0.25}
  \lazyllipse{10}{0}{0.5}{0.25}
    
  \draw (3.5,1.5) -- (3.5,0.5)
  (3,-0.5) -- (2,-1.5)
  (3.5,-0.5) -- (3.5,-1.5)
  (4,-0.5) -- (5,-1.5);
  
  \draw (9.5,-1.5) -- (9.5,-0.5)
  (9,0.5) -- (8,1.5)
  (9.5,0.5) -- (9.5,1.5)
  (10,0.5) -- (11,1.5);
  
  \draw (3.5,-0.25) node{$\large{\blacktriangleleft}$}
  (2.75,0) node{$u$};

  \draw (9.5,-0.25) node{$\large{\blacktriangleleft}$}
  (8.75,0) node{$v$};
  \draw(9,4) node {${\Sigma_u}$};
  \draw(4,-4) node {${\Sigma_v}$};
\end{tikzpicture}
\caption{The graph of spaces $\Udder$. The attaching maps are of
  degree 1 and the black arrows show the orientations.}
\label{fig:1}
\end{figure}
We pick elements $u,v \in \pi_1(\Udder)$ corresponding to the
similarly labelled loops given in Figure \ref{fig:1} and we also
consider groups $\pi_1(\Sigma_u),\pi_1(\Sigma_v)$ to be embedded into
$\pi_1(\Udder)$. 

\begin{defn}
  \label{defn:Magnus-pair}
  Let $G$ be a group, and let $\sim_{\pm}$ be the equivalence relation
  $g\sim_{\pm}h$ if and only if $g$ is conjugate to $h$ or $h^{-1}$, and denote
  by $\left[g\right]_{\pm}$ the $\sim_{\pm}$ equivalence class of
  $g$. A \emph{Magnus pair} is a pair of $\sim_{\pm}$ classes
  $\left[g\right]_{\pm}\neq\left[h\right]_{\pm}$ such that
  $\ncl{g}=\ncl{h}$.
\end{defn}

Note that if $h\in\left[g\right]_{\pm}$ then $\ncl{g}=\ncl{h}$, and
that the relation ``have the same normal closure'' is coarser than
$\sim_{\pm}$, and if a group has a Magnus pair then it is strictly
coarser than $\sim_{\pm}$. To save notation we will say that $g$ and
$h$ are a Magnus pair if their corresponding equivalence classes are.

\begin{lem}\label{lem:uv-Magnus-pair}
  The elements $u$ and $v$ in $\pi_1(\Udder)$ are a Magnus pair.
\end{lem}

\begin{proof}
  The graph of spaces given in Figure \ref{fig:1} gives rise to a
  cyclic graph of groups splitting $D$ of $\pi_1(\Udder)$. The
  underlying graph $X$ has 4 vertices and 8 edges where the vertex
  groups are $\bk{u},\bk{v},\pi_1(\Sigma_u)$, and
  $\pi_1(\Sigma_v)$. Now note that $\pi_1(\Sigma_u)$ can be given the
  presentation
  \[ \pi_1(\Sigma_u) = \bk{a,b,c,d \mid abcd=1} = \bk{a,b,c}\] and
  that the incident edge groups have images
  $\bk{a},\bk{b},\bk{c},\bk{abc} = \bk{d}$. Without loss of generality
  $v^{\pm 1}$ is conjugate to $a$,$b$, and $c$ in $\pi_1(\Udder)$ and
  $u^{\pm 1}$ is conjugate to $d = abc$ in $\pi_1(\Udder)$ which means
  that $u\in \ncl{v}$ and, symmetrically considering $\Sigma_v$,
  $v \in \ncl{u}$.

  On the other hand, the elements $\bk{a},\bk{b},\bk{c},\bk{abc}$ are
  pairwise non-conjugate in $\bk{a,b,c}$ and we now easily see that
  $u$ and $v$ are non-conjugate by considering the action on the
  Bass-Serre tree. $u$ and $v$ therefore form a Magnus pair.
\end{proof}

\subsection{Strict homomorphisms to limit groups}

\begin{defn}\label{defn:QH}
  Let $G$ be a finitely generated group and let $D$ be a
  2-acylindrical cyclic splitting of $G$. We say that a vertex group
  $Q$ of $D$ is \define{quadratically hanging (QH)} if it satisfies
  the following:
  \begin{itemize}
  \item $Q = \pi_1(\Sigma)$ where $\Sigma$ is a compact surface such
    that $\chi(\Sigma) \leq -1$, with equality only if $\Sigma$ is
    orientable or $\bdy(\Sigma)\neq \emptyset$.
  \item The images of the edge groups incident to $Q$ correspond to
    the $\pi_1$-images of $\bdy(\Sigma)$ in $\pi_1(\Sigma)$.
  \end{itemize} 
\end{defn}

\begin{defn}\label{defn:strict}
  Let $G$ be torsion-free group.  A homomorphism $\rho\colon G \to H$ is
  \define{strict} if there some 2-acylindrical abelian
  splitting $D$ of $G$ such that the following hold:
  \begin{itemize}
  \item $\rho$ is injective on the subgroup $A_D$ generated by the
    incident edge groups of each each abelian vertex group $A$ of $D$.
  \item $\rho$ is injective on each edge group of $D$.
  \item $\rho$ is injective on the ``envelope'' $\hat{R}$ of each
    non-QH, non-abelian vertex group $R$ of $D$, where $\hat{R}$ is
    constructed by first replacing each abelian vertex group $A$ of
    $D$ by $A_D$ and then taking $\hat{R}$ to be the subgroup
    generated by $R$ and the centralizers of the edge groups incident
    to $R$.
  \item the $\rho$-images of QH subgroups are non-abelian.
  \end{itemize}
\end{defn}

This next Proposition is a restatement of Proposition 4.21 of
\cite{CG-limits} in our terminology. It is also given as Exercise 8 in
\cite{BF-notes,Wilton-solutions}.

\begin{prop}\label{prop:strict-limit}
  If $L$ is a limit group, $G$ is some finitely generated group such
  that there is a strict homomorphism $\rho: G \to L$, then $G$ is also
  limit group.
\end{prop}

\subsection{$\pi_1(\Udder)$ is a limit group but it is not \frcs.}\label{sec:not-fcs} 

Consider the sequence of continuous maps given in Figure \ref{fig:2}.
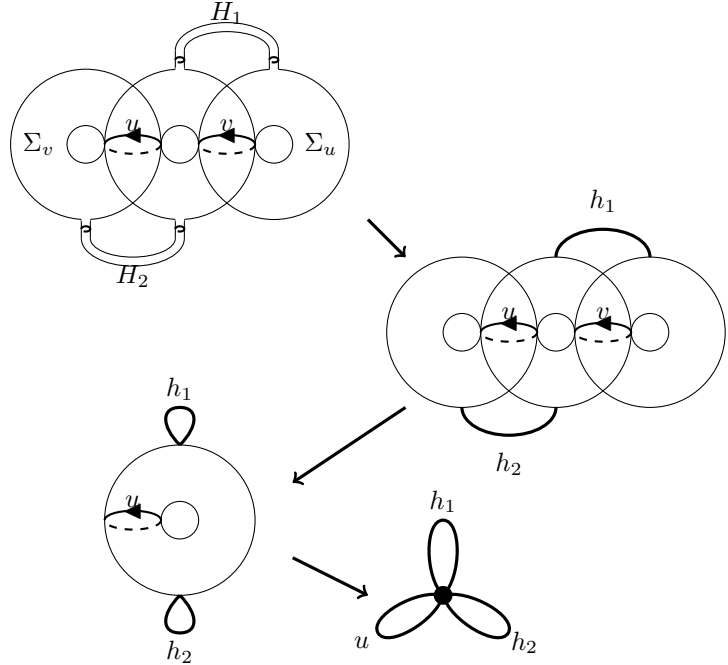
\begin{figure}[htb]
\centering

\begin{tikzpicture}[scale=0.5]
  \draw (-2.5,0) circle (2);
  \draw (-2.5,0) circle (0.5);
  \draw (0,0) circle (2);
  \draw (0,0) circle (0.5);
  \draw (2.5,0) circle (2);
  \draw (2.5,0) circle (0.5);
  \belowcollar{-0.5}{0}{0.75}{0.25}
  \draw (1.25,0.25) node {$\blacktriangleleft$}
  (1.25,0.5) node {$v$};
  \belowcollar{2}{0}{0.75}{0.25}
  \draw (-1.25,0.25) node {$\blacktriangleleft$}
  (-1.25,0.5) node {$u$};
  \draw (-0.12,2) -- (-0.12,2.5);
  \draw (0.12,2) -- (0.12,2.5);
  \blankcircle{(0,2)}{0.12}
  \draw (-2.62,-2) -- (-2.62,-2.5);
  \draw (-2.38,-2) -- (-2.38,-2.5);
  \blankcircle{(2.5,2)}{0.12}
  \draw (-0.12,-2) -- (-0.12,-2.5);
  \draw (0.12,-2) -- (0.12,-2.5);
  \blankcircle{(0,-2)}{0.12}
  \draw (2.38,2) -- (2.38,2.5);
  \draw (2.62,2) -- (2.62,2.5);
  \blankcircle{(-2.5,-2)}{0.12}
  \draw (-0.12,2.5) arc (180:0:1.37 and 0.74);
  \draw (0.12,2.5) arc (180:0:1.13 and 0.5);
  \draw (-0.12,-2.5) arc (360:180:1.13 and 0.5);
  \draw (0.12,-2.5) arc (360:180:1.37 and 0.74);
  \belowcollar{0.12}{2.25}{0.12}{0.07}
  \belowcollar{-2.38}{-2.25}{0.12}{0.07}
  \belowcollar{0.12}{-2.25}{0.12}{0.07}
  \belowcollar{2.62}{2.25}{0.12}{0.07}
  \draw (1.25,3.5) node {$H_1$}
  (-1.25,-3.5) node {$H_2$}
  (3.75,0) node {$\Sigma_u$}
  (-3.75,0) node {$\Sigma_v$};  
  \begin{scope}[shift={(10,-5)}]
    \draw (-2.5,0) circle (2);
    \draw (-2.5,0) circle (0.5);
    \draw (0,0) circle (2);
    \draw (0,0) circle (0.5);
    \draw (2.5,0) circle (2);
    \draw (2.5,0) circle (0.5);
    \belowcollar{-0.5}{0}{0.75}{0.25}
    \draw (1.25,0.25) node {$\blacktriangleleft$}
    (1.25,0.5) node {$v$};
    \belowcollar{2}{0}{0.75}{0.25}
    \draw (-1.25,0.25) node {$\blacktriangleleft$}
    (-1.25,0.5) node {$u$};
    \draw[very thick] (0,2) arc (180:0:1.25 and 0.74);
    \draw[very thick] (0,-2) arc (360:180:1.25 and 0.74);
    \draw (1.25,3.5) node {$h_1$}
    (-1.25,-3.5) node {$h_2$} ;
  \end{scope}
  \begin{scope}[shift={(0,-10)}]
    \draw (0,0) circle (2);
    \draw (0,0) circle (0.5);
    \belowcollar{-0.5}{0}{0.75}{0.25}
    \draw (-1.25,0.25) node {$\blacktriangleleft$}
    (-1.25,0.5) node {$u$};
    \draw[very thick] (0,2)  .. controls (-0.5,2.5) and (-0.5,3)
    .. (0,3) .. controls (0.5,3) and (0.5,2.5) ..(0,2);
    \draw[very thick] (0,-2)  .. controls (-0.5,-2.5) and (-0.5,-3)
    .. (0,-3) .. controls (0.5,-3) and (0.5,-2.5) ..(0,-2);
    \draw (0,3.5) node {$h_1$} (0,-3.5) node {$h_2$};
  \end{scope}
  \begin{scope}[shift={(7,-12)}]
    \draw[very thick] (0,0) .. controls (-0.5,0.5) and (-0.5,2) .. (0,2)
    .. controls (0.5,2) and (0.5,0.5) ..(0,0); \draw[very
    thick,rotate=120] (0,0) .. controls (-0.5,0.5) and (-0.5,2) .. (0,2)
    .. controls (0.5,2) and (0.5,0.5) ..(0,0); \draw[very
    thick,rotate=-120] (0,0) .. controls (-0.5,0.5) and (-0.5,2)
    .. (0,2) .. controls (0.5,2) and (0.5,0.5) ..(0,0);
    \draw (90:2.5) node{$h_1$} 
    (-30:2.5) node{$h_2$}
    (-150:2.5) node{$u$} ; 
    \fill[color=black] (0,0) circle (0.25);
  \end{scope}
  \draw[very thick,->] (5,-2) -- (6,-3);
  \draw[very thick,->] (6,-7) -- (3,-9);
  \draw[very thick,->] (3,-11) -- (5,-12);
\end{tikzpicture}
\caption{A continuous map from $\Udder$ to the wedge of three
  circles. The space on the top left is homeomorphic to $\Udder$. This
  can be seen by cutting along the curves labelled $u,v$.}\label{fig:2}
\end{figure}
The space on the top left obtained by taking three disjoint tori,
identifying them along the longitudinal curves as shown, and then
surgering on handles $H_1,H_2$ is homeomorphic to the space
$\Udder$. A continuous map from $\Udder$ to the wedge of three circles
is then constructed by filling in and collapsing the handles to arcs
$h_1,h_2$, identifying the tori, and then mapping the resulting torus
to a circle so that the image of the longitudinal curve $u$ (or $v$,
as they are now freely homotopic inside a torus) maps with degree 1
onto a circle in the wedge of three circles.

\begin{lem}\label{lem:map-is-nice}
  The homomorphism $\pi_1(\Udder) \rightarrow \freegrp_3$ given by the
  continuous map in Figure \ref{fig:2} is onto, the vertex groups
  $\pi_1(\Sigma_v),\pi_1(\Sigma_u)$ have non-abelian image and the
  edge groups $\bk{u},\bk{v}$ are mapped injectively.
\end{lem}

\begin{proof}
  The surjectivity of the map $\pi_1(\Udder) \rightarrow \freegrp_3$
  as well as the injectivity of the restrictions to $\bk{u},\bk{v}$
  are obvious. Note moreover that the image of $\pi_1(\Sigma_u)$
  contains (some conjugate of) $\bk{u, h_1 u h_1^{-1}}$ and is
  therefore non-abelian, the same is obviously true for the image of
  $\pi_1(\Sigma_v)$.
\end{proof}

The final ingredient is a classical result of Magnus.

\begin{thm}[\cite{Magnus-1931}]\label{thm:Magnus}
  The free group $\free$ has no Magnus pairs.
\end{thm}

\begin{prop}\label{prop:counter-eg}
  $\pi_1(\Udder)$ is a limit group. For every homomorphism
  $\rho\colon\pi_1(\Udder) \rightarrow \freegrp$ the images $\rho(u)$,
  $\rho(v)$ of the elements $u$, $v$ given in Lemma
  \ref{lem:uv-Magnus-pair} are conjugate in $\freegrp$ even though the
  pair $u,v$ are not conjugate in $\pi_1(\Udder)$.
\end{prop}

\begin{proof}
  Lemma \ref{lem:map-is-nice} and Proposition \ref{prop:strict-limit}
  imply that $\pi_1(\Udder)$ is a Limit group. Lemma
  \ref{lem:uv-Magnus-pair} and Theorem \ref{thm:Magnus} imply that,
  for every homomorphism $\pi_1(\Udder) \to \freegrp$ to a free group
  $\freegrp$, the image of $u$ must be conjugate to the image of
  $v^{\pm 1}$ even though $u \not\sim_\pm
  v$. 
\end{proof}

\section{A different failure of \frcsy}

\label{sec:other-case}
 
We now construct another limit group $\idouble$ that is not \frcs, but for a completely different reason.

\begin{defn}[$C$-test words {\cite{ivanov1998certain}}]\label{def:c-test}
  A non-trivial word $w(x_1,\ldots,x_n)$ is a \define{$C$-test word} in $n$
  letters for $\freegrp_m$ if for any two $n$-tuples
  $(A_1,\ldots,A_n), (B_1,\ldots,B_n)$ of elements of $\freegrp_m$
  the equality $w(A_1,\ldots,A_n) = w(B_1,\ldots,B_n) \neq 1$
  implies the existence of an element $S \in \freegrp_m$ such that
  $B_i = SA_i S^{-1}$ for all $i=1,2,\ldots,n.$
\end{defn}

\begin{thm}[{\cite[Main Theorem]{ivanov1998certain}}]\label{thm:c-test}
  For arbitrary $n \geq 2$ there exists a non-trivial indivisible word
  $w_n(x_1,\ldots,x_n)$ which is a $C$-test word in $n$ letters for
  any free group $\freegrp_m$ of rank $m \geq 2$.
\end{thm}

\begin{defn}[Doubles and retractions]
  \label{def:double}
  Let $\freegrp(x,y)$ denote the free group on two generators, let
  $w=w(x,y)$ denote some word in $\{x,y\}^{\pm 1}$. The amalgamated
  free product
  \[ D(x,y;w) = \bk{\freegrp(x,y),\freegrp(r,s)\mid w(x,y)= w(r,s)}\]
  is the \define{double of $\freegrp({x,y})$ along $w$}.  The
  homomorphism $\rho\colon D(x,y;w) \onto \freegrp(x,y)$ given by $r \mapsto
  x, s\mapsto y$ is the \define{standard retraction.}
\end{defn}

\begin{defn}\label{defn:mirror}
  Let $u \in \freegrp(x,y)\leq D(x,y;w)$, but with
  $u \not\sim_\pm w^n$ for any $n$, be given by a specific word
  $u(x,y)$. Its \define{mirror image} is the distinct element
  $u(r,s) \in \freegrp(r,s) \leq D(x,y;w)$. $u(x,y)$ and $u(r,s)$ form
  a \define{mirror pair.}
\end{defn}

It is obvious that mirror pairs are not $\sim_{\pm}$-equivalent. Let
$w$ be a $C$-test word and let $\idouble= D(x,y;w)$. It is well known
that any such double is a limit group. We will call $\idouble$ a
\define{$C$-double}.

\begin{lem}\label{lem:corank2}
  The $C$-double $\idouble$ cannot map onto a free group of rank more than
  $2$.
\end{lem}
\begin{proof}
  $w$ is not primitive in $\freegrp(x,y)$ therefore by
  \cite{shenitzer1955decomposition} $\idouble = D(x,y;w)$ is not free. Theorem
  \ref{thm:c-test} specifically states that $w$ is not a proper
  power. It now follows from \cite[Theorem 1.5]{louder2013scott} that
  $D(w)$ cannot map onto $\freegrp_3$.
\end{proof}

The proof of the next theorem amounts to analyzing a Makanin-Razborov
diagram. We refer the reader to \cite{heil2016jsj} for an explicit
description of this diagram.

\begin{thm}
  For any map $\phi\colon \idouble \to \freegrp$ from a $C$-double
  to some free group, if $u(x,y) \in \freegrp(x,y)$ lies in the
  commutator subgroup $[\freegrp(x,y),\freegrp(x,y)]$, but is not
  conjugate to $w^n$ for any $n$, then the images
  $\phi\left(u(x,y)\right)$ and $\phi\left(u(r,s)\right)$ of mirror
  pairs are conjugate. In particular the limit group $\idouble$ is not
  \frcs.  Furthermore mirror pairs $u(x,y),u(r,s)$ do not form Magnus
  pairs.
\end{thm}

\begin{proof}
  To answer this question we must analyze all maps for $\idouble$ to a free
  group. By Lemma \ref{lem:corank2}, any such map factors through a
  surjection onto $\freegrp_2$, or factors through $\Z$.

  ~ \\ \textbf{Case 1:} \emph{$\phi(w)= 1$.} In this case the factor
  $\freegrp(x,y)$ does not map injectively, it follows that its image
  is abelian. It follows that $\phi$ factors through the free product
  \[
  \pi_{ab}\colon D(x,y;w) \to \freegrp(x,y)^{\mathrm{ab}}*\freegrp(r,s)^{\mathrm{ab}}.
  \]
  In this case all elements of the commutator subgroups of
  $\freegrp(x,y)$ and $\freegrp(r,s)$ are mapped to the identity and
  therefore have conjugate images.
  
  ~\\ \textbf{Case 2:} \emph{$\phi(w)\neq 1.$} In this case the
  factors $\freegrp(x,y),\freegrp(r,s)\leq D(x,y;w)$ map
  injectively. By Theorem \ref{thm:c-test}, since $w$ is a C-test word
  and $\phi(w(x,y)) = \phi(w(r,s)$, there is some $S \in \freegrp_2$
  such that $S\phi(x)S^{-1} =\phi(r)$ and $S\phi(y)S^{-1}
  =\phi(s)$. Suppose now that $w(x,y)$ mapped to a proper power, then
  by \cite[Main Theorem]{Baumslag-1965} $w(x,y) \in \freegrp(x,y)$ is
  part of a basis, which is impossible. It follows that the
  centralizer of $\phi\left(w\right)$ is $\bk{\phi(w)}$ so that
  $S = \phi(w)^n$. Therefore $\phi(r) = w^n\phi(x)w^{-n}$ and
  $\phi(s) = w^n\phi(y)w^{-n}$ and the result follows in this case as
  well.

  We now show that a mirror pair $u(x,y)$ and $u(r,s)$ is not a Magnus
  pair. Consider the quotient $D(x,y;w)/\ncl{u(x,y)}$.  By using a
  presentation with generators and relations, the group canonically
  splits as the amalgamated free product
  \[ \left(\freegrp(x,y)/\ncl{u(x,y)}\right)*_{\bk{\overline w}}
    \left(\freegrp(r,s) / \ncl{w^n}\right)
  \] where $\bk{w^n} = \bk{w} \cap \ncl{u}$ and $\overline{w}$ is the
  image of $w$ in $\bk{w}/\bk{w^n}$.  Now if
  $\ncl{u(x,y)} = \ncl{u(r,s)}$ then we must have
  $D(x,y;w)/\ncl{u(r,s)} = D(x,y;w)/\ncl{u(x,y)}$. This implies
  $\freegrp(r,s)/\ncl{(u(r,s))} = \freegrp(r,s)/\ncl{w^n}$, which
  implies by Theorem \ref{thm:Magnus} that $u(r,s) \sim_\pm w^n$,
  which is a contradiction.
\end{proof}

It seems likely that failure of \frcsy\ should
typically follow from C-test word like behaviour, rather than from
existence of Magnus pairs. 

\section{Towers are \frcs.}\label{sec:towers-fcs}

\begin{defn}\label{defn:quadratic-extension}
  Let $G$ be a group. A \define{regular quadratic extension} of $G$ is
  an extension $G\leq H$ such that \begin{itemize}
  \item $H$ splits as a fundamental group of a graph of groups with
    two vertex groups: $H_{v_1} = G$ and $H_{v_2} = \pi_1(\Sigma)$
    where $H_{v_2}$ is a QH vertex group (See Definition
    \ref{defn:QH}.)
  \item There is a retraction $H \onto G$ such that the image of
    $\pi_1(\Sigma)$ in $G$ is non abelian.
\end{itemize}
We say that $\Sigma$ is the \define{surface associated to the
  quadratic extension}. And note that if $\bdy \Sigma = \emptyset$
then $H = G*\pi_1(\Sigma)$.
\end{defn}

\begin{defn}\label{defn:abelian-extension}
  Let $G$ be a group. An \define{abelian extension by the free
    abelian group $A$} is an extension $G \leq G*_\bk{u} (\bk{u}\oplus
  A) =H$ where $u \in G$ is such that either its centralizer $Z_G(u) =
  \bk{u}$, or $u=1$. In the case where $u=1$ the extension is $G \leq
  G*A$ and it is called a \define{singular abelian extension}.
\end{defn}

\begin{defn}\label{defn:tower}
  Let $\freegrp$ be a (possibly trivial) free group. A \define{tower
    of height $n$ over $\freegrp$} is a group $G$ obtained from a
  sequence of extensions \[ \freegrp = G_0 \leq G_1 \leq \ldots \leq
  G_n = G \] where $G_i \leq G_{i+1}$ is either a regular quadratic
  extension or an abelian extension. The $G_i's$ are the
  \define{levels} of the tower $G$ and the sequence of levels is a
  \define{tower decomposition}. A tower consisting entirely of abelian
  extensions is an \define{iterated centralizer extension.}
\end{defn}

\begin{defn}\label{defn:level-decomposition}
  Let $\freegrp = G_0 \leq \ldots \leq G_n = G$ be a tower
  decomposition of $G$. We call the graphs of groups decomposition of
  $G_i$ with one vertex group $G_{i-1}$ and the other vertex group a
  surface group or a free abelian group as given in Definitions
  \ref{defn:quadratic-extension} and \ref{defn:abelian-extension} the
  \define{$i^{\textrm{th}}$ level decomposition.}
\end{defn}

Towers appear as NTQ groups in the work of Kharlampovich and
Miasnikov, and as $\omega$-residually free towers, as well as
completions of strict resolutions in the the work of Sela. It is a
well known fact that towers are limit groups \cite{KM-IrredI}. This
also follows easily from Proposition \ref{prop:strict-limit} and the
definitions.

\begin{prop}\label{prop:towers-discriminate}
  Let $G$ be a tower of height $n$ over $\freegrp$. Then $G$ is
  discriminated by retractions $G\rightarrow G_{n-1}$. $G$ is also
  discriminated by retractions onto $\freegrp$.
\end{prop}

Following Definition 1.15 of \cite{Sela-Dioph-II} we have:

\begin{defn}\label{defn:closure}
  Let $G$ be a tower. A \define{closure} of $G$ is another tower
  $\cl{G}$ with an embedding $\theta\colon G \into \cl{G}$ such that
  there is a commutative diagram\[
  \begin{tikzpicture}[scale=1.5]
    \node (T0) at (0,0) {$G_0$};
    \node (Ti0) at (0.5,0) {$\leq$};
    \node (T1) at (1,0) {$G_1$};
    \node (Ti1) at (1.5,0) {$\leq$};
    \node (T2) at (2,0) {$\ldots$};
    \node (Ti2) at (2.5,0) {$\leq$};
    \node (T3) at (3,0) {$G_n$};
    \node (Ti3) at (3.5,0) {$=$};
    \node (T4) at (4,0) {$G$};
    
    \node (B0) at (0,-1) {$G_0$};
    \node (Bi0) at (0.5,-1) {$\leq$};
    \node (B1) at (1,-1) {$\cl{G}_1$};
    \node (Bi1) at (1.5,-1) {$\leq$};
    \node (B2) at (2,-1) {$\ldots$};
    \node (Bi2) at (2.5,-1) {$\leq$};
    \node (B3) at (3,-1) {$\cl{G}_n$};
    \node (Bi3) at (3.5,-1) {$=$};
    \node (B4) at (4,-1) {$\cl{G}$};
    
    \draw[->] (T0) -- node[rotate=-90,above]{$=$} (B0);
    \draw[right hook->] (T1) -- (B1);
    \draw[right hook->] (T3) -- (B3);
  \end{tikzpicture}\]
  where the injections $G_i \into \cl{G}_i$ are restrictions of $\theta$
  and the horizontal lines are tower decompositions. Moreover the
  following must hold:
  \begin{enumerate}
  \item If $G_i \leq G_{i+1}$ is a regular quadratic extension with
    associated surface $\Sigma$ such that $\bdy \Sigma$ is
    ``attached'' to $\bk{u_1},\ldots,\bk{u_n} \leq G_i$ then $\cl{G}_i
    \leq \cl{G}_{i+1}$ is a regular quadratic extension with
    associated surface $\Sigma$ such that $\bdy \Sigma$ is
    ``attached'' to $\bk{\theta(u_1)},\ldots,\bk{\theta(u_n)} \leq
    \cl{G}_i$, in such a way that $\theta\colon G_i\into\cl{G}_i$
    extends to a monomorphism $\theta\colon G_{i+1} \into \cl{G}_{i+1}$
    which maps the vertex group $\pi_1(\Sigma)$ surjectively onto the
    vertex group $\pi_1(\Sigma) \leq \cl{G}_{i+1}$.

  \item If $G_i \leq G_{i+1}$ is an abelian extension then $\cl{G}_i
    \leq \cl{G}_{i+1}$ is also an abelian extension. Specifically
    (allowing $u_i=1$) if $G_{i+1} = G_i *_\bk{u_i}(\bk{u_i}\oplus
    A_i)$, then $\cl{G}_{i+1} = \cl{G}_i
    *_\bk{\theta(u_i)}(\bk{\theta(u_i)}\oplus A_i')$. Moreover we require the
    embedding $\theta\colon G_{i+1}\rightarrow \cl{G}_{i+1}$ to map
    $\bk{u_i}\oplus A_i$ to a finite index subgroup of
    $\bk{\theta(u_i)}\oplus A_i'$.
  \end{enumerate}
\end{defn}

We will now state one of the main results of \cite{KM-ift} and
\cite{Sela-Dioph-II} but first some explanations of terminology are in
order. Towers are groups that arise as completed limit groups
corresponding to a strict resolution and the definition of closure
corresponds to the one given in \cite{Sela-Dioph-II}. We also note
that our requirement on the Euler characteristic of the surface pieces
given in Definitions \ref{defn:QH} and \ref{defn:quadratic-extension}
ensures that our towers are coordinate groups of \emph{normalized} NTQ
systems as described in the discussion preceding \cite[Lemma
76]{KM-ift}, we also point out that a \emph{correcting embedding} as
described right before \cite[Theorem 12]{KM-ift} is in fact a
closure in the terminology we are using.

We now give an obvious corollary (in fact a weakening) of
\cite[Theorem1.22]{Sela-Dioph-II}, or \cite[Theorem 12]{KM-ift}; they
are the same result. Let $X,Y$ denote fixed tuples of variables.

\begin{lem}[$\forall\exists$-lifting Lemma]\label{lem:lift}
  Let $\freegrp$ be a fixed non-abelian free group and let
  \[G=\bk{\freegrp,X \mid R(\freegrp,X)}\] be a standard finite
  presentation of a tower over $\freegrp$. Let $W_i(X,Y,\freegrp)=1$
  and $V_i(X,Y,\freegrp)\neq 1$ be (possibly empty) finite systems of
  equations and inequations (resp.) If the following holds:
  \[
  \freegrp \models \forall X \exists Y \Big( R(\freegrp,X) = 1
  \rightarrow \bigvee_{i=1}^m\big(W_i(X,Y,\freegrp)=1 \wedge
  V_i(X,Y,\freegrp)\neq 1 \big) \Big)
  \]
  then there is an embedding $\theta\colon G \into \cl{G}$ into some
  closure such that
  \[
  \cl{G} \models \exists Y
  \bigvee_{i=1}^m\Big(W_i(\theta(X),Y,\freegrp)=1 \wedge
  V_i(\theta(X),Y,\freegrp)\neq 1\Big)
  \]
  where $X$ and $\freegrp$ are interpreted as the corresponding
  subsets of $G = \bk{\freegrp,X \mid R(\freegrp,X)}$
\end{lem}

In the terminology of \cite{Sela-Dioph-II} we have
$G = \bk{\freegrp,X}$ and $\cl{G} = \bk{\freegrp,X,Z}$ for some
collection of elements $Z$. Let $Y = (y_1,\ldots,y_k)$ be a tuple of
elements in $\cl{G}$ that witness the existential sentence above. A
collection of words $y_i(\freegrp,X,Z) =_{G^*} y_i$ is called a set of
\define{formal solution in $\cl{G}$.} According to \cite[Definition
24]{KM-ift} the tuple $Y \subset \cl{G}$ is an \define{$R$-lift}.

\begin{prop}\label{prop:towers-freely-conj-sep}
  Let $G$ be a tower over a non abelian free group $\freegrp $and
  let $S \subset G$ be a finite family of pairwise non-conjugate
  elements of $G$. There exists a discriminating family of retractions
  $\psi_i\colon G\onto \freegrp$ such that for each $\psi_i$ the elements of
  $\psi_i(S)$ are pairwise non-conjugate.
\end{prop}
\begin{proof}
  Suppose towards a contradiction that this was not the case. Then
  either there exists a finite subset $P \subset G\setminus \{1\}$
  such that for every retraction $r\colon G \onto \freegrp$, $1 \in r(P)$ or
  the elements of $r(S)$ are not pairwise non-conjugate. If we write
  elements of $P$ and $S$ as fixed words $\{p_i(\freegrp,X)\}$ and
  $\{s_j(\freegrp,X)\}$ (resp.) then we can express this as a
  sentence. Indeed, consider first the
  formula: \[\Phi_{P,S}(\freegrp,X,t) = \left(\left[ \bigvee_{p_i \in P}
  p_i(\freegrp,X)=1 \right] \vee \left[\bigvee_{(s_i,s_j) \in \Delta(S)}
  t^{-1} s_i(\freegrp,X)t = s_j(\freegrp,X) \right]\right)\] where
  $\Delta(S) = \{(x,y) \in S\times S \mid x \neq y)\}$. In English
  this says that either some element of $P$ vanishes or two distinct
  elements of $S$ are conjugated by some element $t$. We therefore
  have:
  \begin{equation}\label{eqn:formula}\freegrp \models \forall X \left[\left(R(\freegrp,X))=1 \right) \rightarrow \exists
      t\Phi_{P,S}(\freegrp,X,t)\right].
  \end{equation}
  It now follows by Lemma \ref{lem:lift} that there is some closure
  $\theta\colon G \into \cl{G}$ such that \[\cl{G} \models \exists t
  \Phi_{P,S}(\freegrp,\theta(X),t).\] Since $1 \not\in P$ and $\theta$ is
  a monomorphisms none of the $p_i(\freegrp,X)$ are trivial
  so \[\cl{G} \models \exists t \left[ \bigvee_{(s_i,s_j) \in
    \Delta(S)} \left(t^{-1} s_i(\freegrp,X)t = s_j(\freegrp,X)\right) \right].\] In
  particular there are elements $u,v \in G$ which are not conjugate in
  $G$ but are conjugate in $\cl{G}$. We will derive a
    contradiction by showing that this is impossible.

  We proceed by induction on the height of the tower. If the tower has
  height 0 then $G = \freegrp$ and the result obviously holds. Suppose
  now that the claim held for all towers of height $m \leq n$. Let $G$
  have height $n$ and let $u,v$ be non-conjugate elements of $G$ let $G
  \leq \cl{G}$ be any closure and suppose that there is some $t \in
  \cl{G} \setminus G$ such that $t u t^{-1} = v$. 

  Let $D$ be the $n^\nth$ level decomposition of $\cl{G}$ and let $T$
  be the corresponding Bass-Serre tree. Let $T(G)$ be the minimal
  $G$-invariant subtree and let $D_G$ be the splitting induced by the
  action of $G$ on $T(G)$. By Definition \ref{defn:closure} $D_G$ is
  exactly the $n^\nth$ level decomposition of $G$ and two edges of
  $T(G)$ are in the same $G$-orbit if and only if they are in the same
  $\cl{G}$-orbit. We now consider separate cases:
  \\~\\
  {\bf Case 1:} \emph{Without loss of generality $u$ is hyperbolic in the $n^\nth$ level
    decomposition of $G$.}  If $v$ is elliptic in the $n^\nth$ level
  decomposition of $G$ then it is elliptic in the $n^\nth$-level
  decomposition of $\cl{G}$ and therefore cannot be conjugate to $u$
  which acts hyperbolically on $T$.

  It follows that both $u,v$ must be hyperbolic elements with respect to the
  $n^\nth$ level decomposition of $G$. Let $l_u,l_v$ denote the axes
  of $u,v$ (resp.) in $T(G) \subset T$. Since $t u t^{-1} = v$, we
  must have $t\cdot l_u = l_v$. Let $e$ be some edge in $l_u$ then by
  the previous paragraph $t\cdot e \subset l_v$ must be in the same
  $G$-orbit as $e$, which means that there is some $g \in G$ such that
  $gt \cdot e = e$, but again by Definition \ref{defn:closure} the
  inclusion $G\leq \cl{G}$ induces a surjection of the edge groups of
  the $n^\nth$ level decomposition of $G$ to the edge groups of the
  $n^\nth$ level decomposition of $\cl{G}$, it follows that $gt \in G$
  which implies that $t \in G$ contradicting the fact that $u,v$ were
  not conjugate in $G$.
  \\~\\
  {\bf Case 2:} \emph{The elements $u,v$ are elliptic in the $n^\nth$
    level decomposition of $G$.} Suppose first that $u,v$ were
  conjugate into $G_{n-1}$, then the result follows from the fact that
  there is a retraction $G \onto G_{n-1}$ and by the induction
  hypothesis. Similarly by examining the induced splitting of $G \leq
  \cl{G}$, we see that $u$ cannot be conjugate into $G_{n-1}$ and $v$
  into the other vertex group of the $n^\nth$-level decomposition. We
  finally distinguish two sub-cases.
  \\~\\
  {\bf Case 2.1:} \emph{$G_{n-1} \leq G$ is an abelian extension by
    the free abelian group $A$ and $u,v$ are conjugate in $G$ into
    some free abelian group $\bk{w}\oplus A$.} Any homomorphic image
  of $\bk{w}\oplus A$ in $\freegrp$ must lie in a cyclic group, since
  $u \neq v$ in $\cl{G}$ and $\cl{G}$ is discriminated by retractions
  onto $\freegrp$, there must be some retraction $r\colon \cl{G}\rightarrow
  \freegrp$ such that $r(u)\neq r(v)$ which means that $u,v$ are sent
  to distinct powers of a generator of the cyclic subgroup
  $r(\bk{w}\oplus A)$. It follows that their images are not conjugate
  in $\freegrp$ so $u,v$ cannot be conjugate in $\cl{G}$.
  \\~\\
  {\bf Case 2.2:} \emph{$G_{n-1} \leq G$ is a quadratic extension and
    $u$ and $v$ are conjugate in $G$ into the vertex group
    $\pi_1(\Sigma)$.} Arguing as in Case 1 we find that if there is
  some $t \in \cl{G}$ such that $t u t^{-1} = v$ then there is some $g
  \in G$ such that $gt$ fixes a vertex of $T(G) \subset T$ whose
  stabilizer is conjugate to $\pi_1(\Sigma)$. Again by the
  surjectively criterion in item 1. of Definition \ref{defn:closure},
  $gt \in G$ contradicting the fact that $u,v$ were not conjugate in
  $G$. All the possibilities have been exhausted so the result
  follows.
\end{proof}

\begin{proof}[proof of Theorem \ref{thm:generic-sequence}]
  Let $S_1 \subset S_2 \subset S_3 \subset \ldots$ be an exhaustion of
  representatives of distinct conjugacy classes of $G$ by finite
  sets. For each $S_j$ let $\{\psi^j_i\}$ be the discriminating
  sequence given by Proposition \ref{prop:towers-freely-conj-sep}. We take
  $\{\phi_i\}$ to be the diagonal sequence $\{\psi^i_i\}$. This
  sequence is necessarily discriminating and the result follows.
\end{proof}

It is worthwhile to point out that \emph{test sequences} given in the
proof of \cite[Theorem 1.18]{Sela-Dioph-II} or the \emph{generic
  sequence} given in \cite[Definition 44]{KM-ift}, because of their
properties, must satisfy the conclusions of Theorem
\ref{thm:generic-sequence}.  As an immediate consequence of the Sela's
completion construction (\cite[Definition 1.12]{Sela-Dioph-II}) or
canonical embeddings into NTQ groups (\cite[\S 7]{KM-elementary})
Theorem \ref{thm:generic-sequence} implies the following:

\begin{cor}\label{cor:strict-discriminate}
  Let $L$ be a limit group and suppose that for some finite set $S
  \subset L$ there is a homomorphism $f\colon L\to\freegrp$ such
  that: \begin{itemize}
  \item The elements of $f(S)$ are pairwise non-conjugate.
  \item There is a factorization \[f = f_m \circ f_{m-1} \circ \cdots
    \circ f_1\] such that each $f_i$ is a \emph{strict} homomorphisms
    between limit groups (see Definition \ref{defn:strict}).
  \end{itemize}
  Then there is a discriminating sequence $\psi_i\colon L\to \freegrp$ such
  that for all $i$ the elements $\psi_i(S)$ are pairwise non-conjugate.
\end{cor}

\section{Refinements}\label{sec:refinements}

\subsection{$\pi_1(\Udder)$ is almost \frcs.}
\label{sec:almost-fcs}

The limit group $L$ constructed in Section \ref{sec:other-case} had an
abundance of pairs of nonconjugate elements whose images had to have
conjugate images in every free quotient. The situation is completely
different for our Magnus pair group.

\begin{prop}\label{prop:only-u-v}
  $\bk{u},\bk{v} \leq \pi_1(\Udder)$ are the only maximal cyclic
  subgroups of $\pi_1(\Udder)$ whose conjugacy classes cannot be
  separated via a homomorphism to a free group $\pi_1(\Udder) \to
  \freegrp$.
\end{prop}
\begin{proof}
  We begin by embedding $\pi_1(\Udder)$ into a hyperbolic tower. Let
  $\rho\colon \pi_1(\Udder) \onto \freegrp_3$ be the strict homomorphism
  given in Figure \ref{fig:2}. Consider the group
  \[ L =
  \bk{\pi_1(\Udder),\freegrp_3 ,s \mid u = \rho(u), s v s^{-1}= \rho(v)}.
  \]
  This presentation naturally gives a splitting $D$ of $L$ given in
  Figure \ref{fig:3}.
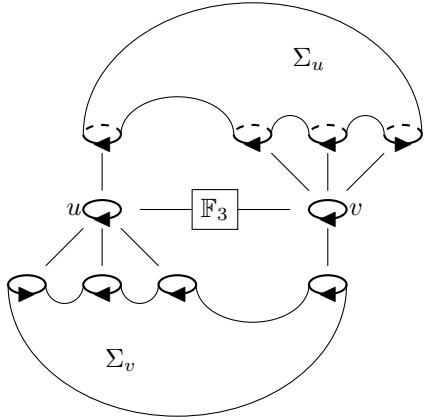
\begin{figure}[htb]
\centering
\begin{tikzpicture}[scale=0.5]
  \abovecollar{4}{2}{0.5}{0.25}
  \draw (3.5,1.75) node {${\blacktriangleleft}$};
  \abovecollar{8}{2}{0.5}{0.25}
  \draw (7.5,1.75) node {${\blacktriangleleft}$};
  \abovecollar{10}{2}{0.5}{0.25}
  \draw (9.5,1.75) node {${\blacktriangleleft}$};
  \abovecollar{12}{2}{0.5}{0.25}
  \draw (11.5,1.75) node {${\blacktriangleright}$};
  
  \lazyllipse{2}{-2}{0.5}{0.25}
  \draw (1.5,-2.25) node {${\blacktriangleright}$};
  \lazyllipse{4}{-2}{0.5}{0.25}
  \draw (3.5,-2.25) node {${\blacktriangleleft}$};
  \lazyllipse{6}{-2}{0.5}{0.25}
  \draw (5.5,-2.25) node {${\blacktriangleleft}$};
  \lazyllipse{10}{-2}{0.5}{0.25}
  \draw (9.5,-2.25) node {${\blacktriangleleft}$};

  \draw (3,2) arc (180:0:4.5 and 3.5);
  \draw (4,2) arc (180:0:1.5 and 1);
  \draw (8,2) arc (180:0:0.5 and 0.5);
  \draw (10,2) arc (180:0:0.5 and 0.5);

  \draw (10,-2) arc (360:180:4.5 and 3.5);
  \draw (9,-2) arc (360:180:1.5 and 1);
  \draw (5,-2) arc (360:180:0.5 and 0.5);
  \draw (3,-2) arc (360:180:0.5 and 0.5);
  
  \lazyllipse{4}{0}{0.5}{0.25}
  \lazyllipse{10}{0}{0.5}{0.25}
    
  \draw (3.5,1.5) -- (3.5,0.5)
  (3,-0.5) -- (2,-1.5)
  (3.5,-0.5) -- (3.5,-1.5)
  (4,-0.5) -- (5,-1.5);
  
  \draw (9.5,-1.5) -- (9.5,-0.5)
  (9,0.5) -- (8,1.5)
  (9.5,0.5) -- (9.5,1.5)
  (10,0.5) -- (11,1.5);
  
  \draw (3.5,-0.25) node{$\large{\blacktriangleleft}$}
  (2.75,0) node{$u$};

  \draw (9.5,-0.25) node{$\large{\blacktriangleleft}$}
  (10.25,0) node{$v$};
  \draw(9,4) node {${\Sigma_u}$};
  \draw(4,-4) node {${\Sigma_v}$};
  \draw (6,0) -- (4.5,0);
  \draw (7,0) -- (8.5,0);
  \node (a) at (6.5,0) [rectangle,draw,fill=white] {$\freegrp_3$};
\end{tikzpicture}
\caption{The splitting $D$ of $L$.}
\label{fig:3}
\end{figure}
We have a retraction $\rho*\colon L \onto \freegrp_3$ given by\[
\rho*\colon \left\{\begin{array}{l}
    g \mapsto \rho(g); g \in \pi_1(\Udder)\\
    f \mapsto f; f \in \freegrp_3\\
    s \mapsto 1\\
  \end{array}\right.
\] It therefore follows that $L$ is a hyperbolic tower over $\freegrp_3$.

Claim: \emph{if $\alpha, \beta \in \pi_1(\Udder) \leq L$ are
  non-conjugate in $\pi_1(\Udder)$ and $\alpha,\beta$ are not both
  conjugate to $\bk{u}$ or $\bk{v}$ in $\pi_1(\Udder)$ then they are
  not conjugate in $L$.} If both $\alpha$ and $\beta$ are elliptic,
then this follows easily from the fact that the vertex groups are
malnormal in $L$. Also $\alpha$ cannot be elliptic while $\beta$ is
hyperbolic. Suppose now that $\alpha,\beta$ are hyperbolic. Let $T$ be
the Bass-Serre tree corresponding to $D$ and let $T' =
T(\pi_1(\Udder))$ be the minimal $\pi_1(\Udder)$ invariant
subtree. Suppose that there is some $s \in L$ such that $s \alpha
s^{-1} = \beta$, then as in the proof of Proposition
\ref{prop:towers-freely-conj-sep} and Proposition
we find that for some $g \in \pi_1(\Udder)$ either $gs$ permutes two
edges in $T'$ that are in distinct $\pi_1(\Udder)$-orbits or it fixes
some edge in $T'$. The former case is impossible and it is easy to see
that the latter case implies that $gs \in \pi_1(\Udder)$. Therefore we
have a contradiction to the assumption that $\alpha,\beta$ are not
conjugate in $\pi_1(\Udder)$. The claim is now proved.

It therefore follows that if $\alpha,\beta \in \pi_1(\Udder) \leq L$
are as above, then by Theorem \ref{thm:generic-sequence} there exists
some retraction $r\colon L \onto \freegrp_3$ such that $r(\alpha),r(\beta)$
are non-conjugate. 
\end{proof}

This construction gives an alternative proof to the fact that
$\pi_1(\Udder)$ is a limit group. The group $L$ constructed is a
triangular quasiquadratic group and the retraction $\rho^*$ makes it
non-degenerate, and therefore an NTQ group. $L$ and therefore
$\pi_1(\Udder)\leq L$ are therefore limit groups by \cite{KM-IrredI}.

\subsection{$C$-doubles do not contain Magnus pairs.}\label{sec:no-mag-pairs}

Theorem \ref{thm:generic-sequence} enables us to examine a $C$-double
$\idouble$ more closely.

\begin{prop}\label{prop:ivanov-double-no-magnus}
  The $C$-double $\idouble$ constructed in Section
  \ref{sec:other-case} does not contain a Magnus pair.
\end{prop}
\begin{proof}
  We need to show that if two elements $u,v$ of $\idouble$ have the
  same normal closure in $\idouble$ then they must be
  conjugate. Suppose that $u,v$ are both elliptic with respect to the
  splitting (as a double) of $\idouble$ but not conjugate. By Theorem
  \ref{thm:c-test} if they are conjugate to a mirror pair $(u^g,v^h)$
  for some $g,h \in \idouble$ then they do not form a Magnus pair,
  i.e. they have separate normal closures. Otherwise there are
  homomorphisms $\idouble \to \freegrp$ in which $u,v$ have
  non-conjugate images, therefore by Theorem \ref{thm:Magnus} the
  normal closures of their images are distinct; so
  $\ncl{u} \neq \ncl{v}$ as well.

  Suppose now that $u$ or $v$ is hyperbolic in $\idouble$. Recall the
  generating set $x,y,r,s$ for $\idouble$ given in Definition
  \ref{def:double}. Let $\freegrp = \freegrp(x,y)$ and consider the
  embedding into a centralizer extension, represented as an HNN
  extension
  \begin{align*}
    \idouble & \into \bk{\freegrp,t | t w(x,y) = w(x,y)t} = \freegrp*^t_\bk{w}\\
    x & \mapsto x,~~~ y  \mapsto y\\
    r & \mapsto t^{-1}xt,~~~ s  \mapsto t^{-1}yt
  \end{align*}
  The stable letter $t$ makes mirror pairs conjugate in this bigger
  group.  A hyperbolic element of $\idouble$ can be written as a product
  of syllables\[ u = a_1(x,y)a_2(r,s)\cdots a_l(r,s)
  \] with $a_1$ or $a_l$ possibly trivial. The image of $u$ in
  $\freegrp*^t_\bk{w}$ is \[
    u = a_1(x,y)\left(t^{-1}a_2(x,y)t\right)\cdots
    \left(t^{-1}a_l(x,y)t\right).\] Consider the set of words of the
  form  \[
    w_1(x,y)\left(t^{-1}w_2(x,y)t\right)\cdots w_{N-1}(x,y)\left(t^{-1}w_{N}(x,y)t\right),
  \] with $w_1$ or $w_N$ possibly trivial. This set is clearly closed
  under multiplication, inverses and passing to
  $\freegrp_{\bk w}^t$-normal form. It follows that we can identify
  the image of $\idouble$ with this set of words, which we call
  \define{$t^{-1}*t$-syllabic words}. Each factor $w_i(x,t)$ or
  $t^{-1}w_j(x,y)t$ is called a \define{$t^{-1}*t$-syllable}.

  It is an easy consequence of Britton's Lemma that if $u$ is a
  hyperbolic, i.e. with cyclically reduced syllable length more than
  1, $t^{-1}*t$-syllabic word and $g^{-1}ug$ is again
  $t^{-1}*t$-syllabic for some $g$ in $\freegrp*_\bk{w}^t$ then $g$
  must itself be $t^{-1}*t$-syllabic. Indeed this can be seen by
  cyclically permuting the $\freegrp*_\bk{w}^t$-syllables of a
  cyclically reduced word $u$. We refer the reader to \cite[\S
  IV.2]{Lyndon-Schupp-1977} for further details about normal forms and
  conjugation in HNN extensions.

  Suppose now that $u,v$ are non conjugate in $\idouble$, but have the
  same normal closure in $\idouble$. Since at least one of them is
  hyperbolic in $\idouble$, it is clear from the embedding that its
  image must also be hyperbolic with respect to the HNN splitting
  $\freegrp*_\bk{w}^t$. Now, since
  $\ncl{u}_\idouble = \ncl{v}_\idouble$, in the bigger group
  $\freegrp*_\bk{w}^t$ we
  have:\[ \ncl{u}_{\freegrp*_\bk{w}^t} =
    \ncl{\ncl{u}_\idouble}_{\freegrp*_\bk{w}^t} =
    \ncl{\ncl{v}_\idouble}_{\freegrp*_\bk{w}^t} =
    \ncl{v}_{\freegrp*_\bk{w}^t} \]

  By Theorem \ref{thm:generic-sequence} or \cite{Lioutikova-CRF}
  centralizer extensions are freely conjugacy separable, therefore
  they cannot contain Magnus pairs. It follows that $u,v$ must be
  conjugate in the bigger $\freegrp*_\bk{w}^t$. Let
  $g^{-1}ug=_{\freegrp*_{\bk w}^t} v$. Now both $u$ and $v$ must be
  hyperbolic so it follows that $g$ must also be a $t^{-1}*t$-syllabic
  word; thus $g$ is in the image of $\idouble$ of
  $\freegrp*_{\bk w}^t$. Furthermore since the map
  $\idouble \into \freegrp*_{\bk w}^t$ is an
  embedding\[ g^{-1}ug=_{\freegrp*_{\bk w}^t} v \Rightarrow
    g^{-1}ug=_\idouble v,
  \] contradicting the fact that $u,v$ are non conjugate in $\idouble$.
\end{proof}

\subsection{A non-tower limit group that is \frcs}
\label{sec:tower-non-eg}

In this section we construct a limit group that is \frcs\ but which
does not admit a tower structure. Let $H \leq [\freegrp,\freegrp]$ be
some f.g. malnormal subgroup of $\freegrp$, e.g.
$H = \bk{aba^{-1}b^{-1},b^{-2}a^{-1}b^2a} \leq \freegrp(a,b)$. And
pick $h \in H \setminus [H,H]$ such that $H$ is \define{rigid}
relative to $h$, i.e. $H$ has no non-trivial cyclic or free splittings
relative to $\bk{h}$.  Because $h \in [\freegrp,\freegrp]$ there is is
a quadratic extension \[\freegrp < \freegrp*_\bk{h} \pi_1(\Sigma) \]
where $\Sigma$ has one boundary component and has genus
$g = \textrm{genus}(h)$, in particular there is a retraction onto
$\freegrp$. Consider now the subgroup $L = H*_\bk{h}\pi_1(\Sigma)$.

\begin{prop}
  $L$ as above is \frcs.
\end{prop}
\begin{proof}
  Because $H \leq \freegrp$ was chosen to be malnormal, an easy
  Bass-Serre theory argument (e.g. apply \cite[Theorem
  IV.2.8]{Lyndon-Schupp-1977}) tells us that $\alpha,\beta \in L$ are
  conjugate if and only if they are conjugate in
  $\freegrp*_\bk{h} \pi_1(\Sigma)$. On the other hand by Theorem
  \ref{thm:generic-sequence}, $\freegrp*_\bk{h} \pi_1(\Sigma)$, and
  hence $L$, are \frcs.
\end{proof}

\begin{defn}
  A splitting $\mathbb X$ is \emph{elliptic} in a splitting
  $\mathbb Y$ if every edge group in $\mathbb X$ is conjugate to a
  vertex group of $\mathbb Y$. Otherwise we say $\mathbb X$ is
  hyperbolic in $\mathbb Y$.
\end{defn}

\begin{thm}[{\cite[Theorem 7.1]{R-S-JSJ}}]\label{thm:JSJ} Let $G$ be an
    f.p. group with a single end. There exists a reduced, unfolded
    $\Z$-splitting of $G$ called a JSJ decomposition of $G$ with the
    following properties:
  \begin{enumerate}
  \item\label{it:jsj-cmq} Every canonical maximal QH (recall definition
    \ref{defn:QH}) subgroup (CMQ) of $G$ is conjugate to a vertex
    group in the JSJ decomposition. Every QH subgroup of $G$ can be
    conjugated into one of the CMQ subgroups of $G$. Every non-CMQ
    vertex groups in the JSJ decomposition is elliptic in every
    $\Z$-splitting of $G$.
  \item\label{it:jsj-hyphyp}An elementary $\Z$-splitting $G = A*_CB$ or $G=A*_C$ which
    is hyperbolic in another elementary $\Z$-splitting is obtained
    from the JSJ decomposition of $G$ by cutting a 2-orbifold
    corresponding to a CMQ subgroup of $G$ along a weakly essential
    simple closed curve (s.c.c.).
  \item\label{it:jsj-ell} Let $\Theta$ be an elementary $\Z$-splitting $G = A*_CB$ or $G=A*_C$ which
    is elliptic with respect to any other elementary $\Z$ splitting of
    $G$. There exists a $G$-equivariant simplicial map between a
    subdivision of $T_{\mathrm{JSJ}}$, the Bass-Serre tree
    corresponding to the JSJ decomposition, and $T_\Theta$, the
    Bass-Serre tree corresponding to $\Theta$.
  \item\label{it:jsj-to-general} Let $\Lambda$ be a general $\Z$-splitting of $G$. There exists
    a $\Z$-splitting $\Lambda_1$ obtained from the JSJ decomposition
    by splitting the CMQ subgroups along weakly essential s.c.c. on
    their corresponding 2-orbifolds, so that there exists a
    $G$-equivariant simplicial map between a subdivision of the
    Bass-Serre tree $T_{\Lambda_1}$ and $T_{\Lambda}.$
  \item\label{it:jsj-canonical} If $\mathrm{JSJ}_1$ is another JSJ
    decomposition of $G$, then there exists a $G$-equivariant
    simplicial map $h_1$ from a subdivision of $T_{\mathrm{JSJ}}$ to
    $T_{\mathrm{JSJ}_1}$, and a $G$-equivariant simplicial map $h_2$
    from a subdivision of $T_{\mathrm{JSJ}_1}$ to $T_{\mathrm{JSJ}}$,
    so that $h_1\circ h_2$ and $h_2 \circ h_1$ are $G$-homotopic to
    the corresponding identity maps.
  \end{enumerate}
\end{thm}

We note that item \ref{it:jsj-canonical}. of the above theorem
describes the canonicity of a JSJ decomposition.

\begin{lem}
  The splitting $L = H*_\bk{h}\pi_1(\Sigma)$ is a cyclic JSJ
  splitting.
\end{lem}
\begin{proof}
  This is an elementary $\Z$ splitting of $L$, let's see how it can be
  obtained from the JSJ decomposition given in Theorem
  \ref{thm:JSJ}. The first case is if $h$ is elliptic in every other
  splitting. Then by \ref{it:jsj-ell}. of Theorem \ref{thm:JSJ} there exists an $L$-
  equivariant map $\rho$ from $T_{\mathrm{JSJ}}$ to the Bass-Serre
  tree $T$ corresponding to $H*_\bk{h}\pi_1(\Sigma)$ in which $H$
  stabilizes a vertex $v$. It follows that $H$ acts on
  $\phi^{-1}(\{v\}) = T_H \subset T_{\mathrm{JSJ}}$. Since $H$ is
  rigid relative to $h$ and $h$ acts elliptically on
  $T_{\mathrm{JSJ}}$, $T_H$ cannot be infinite, since that would imply
  that $H$ admits an essential cyclic splitting relative to $h$. $T_H$
  must in fact be a point. Otherwise $T_H$ is a finite tree tree and
  there must be a ``boundary'' vertex $u\in T_H$ such that
  $H \not \geq L_u$. Since $\phi(u) = v$, $L$-equivariance implies
  that $L_u$ fixes $v$ so that $L_u \leq H$, which is a
  contradiction. It follows that in this case $H$ is actually a vertex
  group of the JSJ decomposition and $\pi_1(\Sigma)$ must be a CMQ
  vertex group.

  The second case is that $h$ is hyperbolic in some other
  $\Z$-splitting $\mathbb D$ of $L$.  Since $H$ is rigid relative to
  $h$, $H$ must be hyperbolic with respect to $\mathbb D$. Now by
  \ref{it:jsj-hyphyp}. of Theorem \ref{thm:JSJ} the splitting
  $L = H*_\bk{h}\pi_1(\Sigma)$ can be obtained from the JSJ splitting
  of $L$ by cutting along a simple closed curve on some CMQ vertex
  group, and this curve is conjugate to $h$. But this means that $H$
  admits a cyclic splitting as a graph of groups with a QH vertex
  group $\pi_1(\Sigma')$ such that the $\pi_1$-image of some connected
  component of $\bdy \Sigma'$ is conjugate to $\bk{h}$, in particular
  $H$ must have a cyclic splitting relative to $h$, which contradicts
  the fact that $H$ is rigid relative to $h$.
\end{proof}

\begin{prop}\label{prop:not-a-tower}
  The limit group $L = H*_\bk{h}\pi_1(\Sigma)$ does not admit a tower
  structure.
\end{prop}
\begin{proof}

  Suppose towards a contradiction that $L$ was a tower, consider the
  last level:
  \[L_{n-1} < L_n = L.\] Since $L$ has no non-cyclic abelian subgroups
  $L_{n-1} < L$ must be a hyperbolic extension. This means that $L$
  admits a cyclic splitting $\mathbb D$ with a vertex group $L_{n-1}$
  and a QH vertex group $Q$. Since $L = H*_\bk{h}\pi_1(\Sigma)$ is a
  JSJ decomposition and $\pi_1(\Sigma)$ is a CMQ vertex group. By
  \ref{it:jsj-cmq}. and \ref{it:jsj-to-general}. of Theorem
  \ref{thm:JSJ}, the QH vertex group $Q$ must be represented as
  $\pi_1(\Sigma_1)$, where $\Sigma_1$ is a connected subsurface
  $\Sigma_1 \subset \Sigma$. It follows from
  \ref{it:jsj-to-general}. of Theorem \ref{thm:JSJ} that the other
  vertex group must be $L_{n-1} = H*_{\bk{h}} \pi_1(\Sigma')$ where
  $\Sigma' = \Sigma \setminus \Sigma_1$.

  Since $L_{n-1} < L$ is a quadratic extension there is a retraction
  $L \onto L_{n-1}$. Note however that because $\Sigma'$ has at least
  two boundary components \[H*_{\bk{h}} \pi_1(\Sigma') =
  H*\freegrp_m\] where $m = -\chi(\Sigma')$. Now since we have a
  retraction $L \onto L_{n-1}$ there is are $x_i,y_i \in L_{n-1}$ such
  that \[h = \prod_{i=1}^g[x_i,y_i]\] But this would imply that $h \in
  [L_{n-1},L_{n-1}]$ which is clearly seen to be false by abelianizing
  $H*\freegrp_m$ and remembering that $h \not\in [H,H]$.
\end{proof}

\bibliographystyle{alpha} \bibliography{conj-res-free.bib}
\end{document}